\documentclass[11pt]{article}
\usepackage{amsmath,amsthm,amsfonts,amssymb}

\DeclareMathOperator{\NN}{\mathbb N}

\DeclareMathOperator{\C}{\mathbb C}

\newcommand{\E}{\mathcal E}
\renewcommand{\C}{\mathcal C}

\newcommand{\Ic}{{\rm Incr}(\NN)}

\newcommand{\eps}{\varepsilon}
\newcommand{\ca}[1]{{\mathcal #1}}

\theoremstyle{plain}
\newtheorem{theorem}{Theorem}
\newtheorem{lemma}[theorem]{Lemma}
\newtheorem{proposition}[theorem]{Proposition}
\newtheorem{corollary}[theorem]{Corollary}
\newtheorem{conjecture}[theorem]{Conjecture}

\newtheorem{question}[theorem]{Question}
\newtheorem{fact}[theorem]{Fact}
\theoremstyle{definition}
\newtheorem{remark}[theorem]{Remark}
\newtheorem{definition}[theorem]{Definition}
\newtheorem{example}[theorem]{Example}
\newtheorem{exercise}[theorem]{Exercise}
\newtheorem{problem}{Problem}

\numberwithin{theorem}{section}
\numberwithin{equation}{section}

\newcommand{\bt}{\begin{theorem}}
\newcommand{\et}{\end{theorem}}

\newcommand{\bl}{\begin{lemma}}
\newcommand{\el}{\end{lemma}}

\newcommand{\bd}{\begin{definition}}
\newcommand{\ed}{\end{definition}}

\newcommand{\beq}{\begin{equation}}
\newcommand{\eeq}{\end{equation}}

\newcommand{\bexa}{\begin{example}}
\newcommand{\eexa}{\end{example}}

\newcommand{\bexe}{\begin{exercise}}
\newcommand{\eexe}{\end{exercise}}

\newcommand{\bfact}{\begin{fact}}
\newcommand{\efact}{\end{fact}}

\newcommand{\bprop}{\begin{proposition}}
\newcommand{\eprop}{\end{proposition}}

\newcommand{\bp}{\begin{proof}}
\newcommand{\ep}{\end{proof}}

\newcommand{\bc}{\begin{corollary}}
\newcommand{\ec}{\end{corollary}}

\newcommand{\bq}{\begin{question}}
\newcommand{\eq}{\end{question}}

\newcommand{\bconj}{\begin{conjecture}}
\newcommand{\econj}{\end{conjecture}}
\newcommand\Om{\Omega}
\newcommand{\br}{\begin{remark}}
\newcommand{\er}{\end{remark}}

\newcommand{\bproblem}{\begin{problem}}
\newcommand{\eproblem}{\end{problem}}

\newcommand{\ben}{\begin{enumerate}}
\newcommand{\een}{\end{enumerate}}

\renewcommand{\phi}{\varphi}

\newcommand{\si}{\mathcal S}

\newcommand{\forae}{{\widetilde\forall}}

\begin{document}

\title{Monotone paths in random hypergraphs} 

\author{{\bf P. Majer}
        \thanks{Dipartimento di Matematica, Universit\`a di Pisa, 
        Largo B. Pontecorvo 5, 56127 Pisa, Italy, 
        e-mail: majer@dm.unipi.it},
        \ {\bf M. Novaga}
        \thanks{Dipartimento di Matematica, Universit\`a di Padova,
        via Trieste 63, 35121 Padova, Italy, e-mail: novaga@dm.unipi.it}}

\date{} 

\maketitle

\begin{abstract}
\noindent We determine the probability thresholds for the existence of monotone paths,
of finite and infinite length, in random oriented graphs with vertex set $\NN^{[k]}$,
the set of all increasing $k$-tuples in $\NN$.
These graphs appear as line graph of uniform hypergraphs with vertex set $\NN$.
\end{abstract}

\section{Introduction}

In this paper we are interested in oriented graphs with vertex set $\NN^{[k]}$, the set 
of all increasing $k$-tuples of elements of $\NN$.
We recall that an oriented graph $G$ is a pair of sets $(V_G,E_G)$, where $V_G$ is the set of vertices  
and $E_G\subseteq V_G\times V_G$ is the set of edges of $G$, such that $(a,b)\in E_G$ implies 
$(b,a)\not\in E_G$ for all $a,b\in V_G$. 

We point out that such graphs naturally arise as line graphs of
$k$-uniform hypergraphs with vertex set $\NN$ \cite{VV:09}. 
Indeed, every $k$-edge of such a hypergraph 
is an element of $\NN^{[k]}$, and we can uniquely associate to the
hypergraph the line graph with vertex set $\NN^{[k]}$, where two vertices are linked if the two
corresponding hyperedges have intersection of a prescribed type. 

A random subgraph of $G$ is a random choice of its vertices:
more precisely, we associate to each vertex $v\in V_G=\NN^{[k]}$ a measurable set $X_{v}\subseteq \Om$,
where $(\Om, \mu)$ is a given probability space. We recall that $G$
can be equivalently defined by means of a $\mu$-measurable function 
$X: \Om \to 2^{V_G}$, such that $X_v := \{\omega\in \Om:\, v \in X(\omega)\}$
for all $v\in V_G$. 

Notice that we consider graphs with random vertices rather than random edges, 
as this class of graphs seems more natural for the 
questions we address in this paper. However, for the sake of completeness,
in Section \ref{randomed} we briefly discuss the case of random edges.

The main problem we study in this paper is whether 
a random subgraph $X(\omega)$ of $G$ contains an infinite path for some $\omega\in \Om$.
\bproblem \label{problemv} %Let $G$ be an oriented graph, with $V_G=\NN^{[k]}$ for some $k\in\NN$.
For all $v \in V_G=\NN^{[k]}$ let $X_{v}$ be a measurable subset 
of a probability space $(\Om, \mu)$, with $\mu(X_{v})\ge \lambda\in [0,1]$.
We ask for which values of $\lambda$
there exists an infinite sequence of vertices $v_i$ of $G$ 
such that $(v_i,v_{i+1})$ is an edge of $G$ for all $i\in\NN$ 
and $\bigcap_{i\in \NN} X_{v_i}$ is non-empty.
\eproblem

In order to answer to Problem \ref{problemv},
we have to compute the threshold 
\begin{equation}\label{eqlac}
\lambda_G:=\sup\left\{\inf_{v\in V_G=\NN^{[k]}}\mu(X_{v}):\ X \ \textrm{random graph without infinite paths}\right\}
\end{equation}
and we also want to estimate the probability 
\[
\mu(\{ \omega\in \Om: X(\omega)\textrm{ contains an infinite path}\})
\]  
that the random graph $X$ contains an infinite path, in terms of the parameter $\lambda$. 
%that bounds from below the probability that a vertex $\alpha$ 
%belongs to the random subgraph, i.e. $\mu(X_{\alpha})\ge \lambda$ for all $\alpha\in V_G$. 

\smallskip

Problem \ref{problemv} was first posed in \cite{EH:64} for the graphs $G_k$ defined in Section \ref{seccong}, that is, 
$V_{G_k}=\NN^{[k]}$ and $(v_1,v_2)\in E_{G_k}$, with $v_1=(i_0,\ldots,i_{k-1})$ and $v_2=(j_0,\ldots,j_{k-1})$, if and only if
$j_\ell=i_{\ell+1}$ for all $\ell\in \{0,\ldots,k-2\}$. 
A complete solution of Problem \ref{problemv} for $k=2$ ($k=1$ being trivial) was given in \cite{FT:85}, 
where the authors show that an infinite path necessarily exists if $\lambda\ge 1/2$, while there are counterexamples for $\lambda<1/2$.
In the paper \cite{BMN:12} we present a different proof of this result, relying on the theory of exchangeable measures 
introduced by De Finetti in \cite{DF:30} and later developed by many other authors (see \cite{A:08} and references therein).

\smallskip

In this paper, following the approach of \cite{BMN:12}, we provide a solution  
to Problem \ref{problemv} for a class of graphs which includes 
the graphs $G_k$ for all $k\in\NN$ (see Theorem \ref{teoromitov}). 

An important notion in this paper is the  
{\it joint contractability} of measures (see Section \ref{secconm} and \cite{K:92, K:05}), which extends the notion of exchangeability 
by demanding a smaller class of symmetries.
For such measures there is a remarkable representation theorem (Theorem \ref{teoah})
due to Kallenberg \cite{K:92}, which is crucial for the present work.
Using this representation, we are able to reformulate Problem \ref{problemv} as a variational problem 
on the unit cube, which we then solve by combinatorial methods (see Lemma \ref{lemvar}).

\smallskip

As discussed in detail in Section \ref{secphi},
if a random subgraph $X$ of $G$ has no infinite paths, one can define a 
$\mu$-measurable map $\varphi \colon \Omega \to \omega_1^{\NN^{[k]}}$ where $\varphi(\omega,v)$ is the rank of the vertex $v\in V_G=\NN^{[k]}$ 
in the graph $X(\omega)$. 
It turns out that $\phi_*(\mu)$ is a compactly supported Borel measure on $\omega_1^{\NN}$,
and that 
$$
\phi(X_v)\subseteq 
A_v:=\left\{ x\in \omega_1^{\NN^{[k]}}:\, x_v>x_{v'}\, {\rm for\ all}\, v'\, {\rm such\, that}\,(v,v')\in E_G\right\}.
$$ 
As a consequence, in the determination of the threshold $\lambda_G$,
we can set $\Omega=\omega_1^{\NN^{[k]}}$, $X_{v}=A_{v}$, and reduce to 
a variational problem on the convex set ${\mathcal M}^1_c(\omega_1^{\NN^{[k]}})$ of compactly supported probability 
measures on $\omega_1^{\NN^{[k]}}$:
\begin{equation}\label{probvar}
\lambda_G = \sup_{m\in {\mathcal M}^1_c(\omega_1^{\NN^{[k]}})}\inf_{v\in \NN^{[k]}}\,m\left( A_{v}\right).
\end{equation}
As a next step, we show that in \eqref{probvar} we can equivalently take the supremum in the smaller class 
of all the compactly supported \emph{contractable measures} on $\omega_1^{\NN^{[k]}}$.
Thanks to this reduction, we can explicitly compute $\lambda_G$ when $G$ is a simple contractible graph (see Theorem \ref{teoromitov}).
%We note that the supremum in \eqref{probvar} is not attained, which implies that for $\mu(X_{i,j})\ge 1/2$ infinite paths occur with positive probability. 

\smallskip

A closely related problem, which we discuss at the end of the paper, is finding the thresholds 
for paths of finite length. We observe that this type of questions 
are reminiscent of Tur\'an's problem \cite{DC:83}, 
and the answer is expected to be difficult in general.
We mention that the case 
of paths of length $2$ in the graphs $G_k$ mentioned above
has been already considered in \cite{TW:98}, 
where the authors determine the precise thresholds for $k\in \{ 2,3\}$, 
show upper and lower bounds for all $k\in\NN$, and make a conjecture for $k\in \{ 4,5\}$. 
In Section \ref{secopen}, we determine the thresholds for all $k$ odd, 
thus confirming the conjecture made in \cite{TW:98} for $k=5$, while 
the case of $k$ even remains open. 

%%%%%%%%%%%%%%%%%%%%%%%%%%%%%%%%%%%%%%%%%%%%%%%%%%%%%%%%%%%%%%%%%%%%%%%%%%%%%%%%%%%%%
\section{Notation}\label{secnot}
%%%%%%%%%%%%%%%%%%%%%%%%%%%%%%%%%%%%%%%%%%%%%%%%%%%%%%%%%%%%%%%%%%%%%%%%%%%%%%%%%%%%%

Given $p\in \NN$, we identify $p$ with the set $\{0,\ldots, p-1\}$. 
Given a topological space $S$ and $k\in\NN$, we let $S^{[k]}$ be the set 
of all subsets of $S$ of cardinality $k$, endowed with the product topology, and we let 
$S^{[\le k]}$ be the set of all subsets of $S$ of cardinality at most $k$.
Notice that, if $S$ has cardinality $k\in\NN$, then $S^{[\le k]}=S^{[\le j]}$ for all $j>k$ and
$S^{[\le k]}=\mathcal P(S)$ is the part set of $S$. 
If $S$ is ordered, we can identify $S^{[k]}$ with the set of 
$k$-tuples $(i_0, \ldots, i_{k-1})$ with $i_0<\cdots < i_{k-1}$. 

Given a compact metric space $\Lambda$,
we let $\Lambda^{\NN^{[k]}}$ be the space of all sequences with indices in $\NN^{[k]}$ and values in $\Lambda$, 
endowed with the product topology, and 
$\mathcal M(\Lambda^{\NN^{[k]}})$ be the space of all Radon measures on $\Lambda^{\NN^{[k]}}$, 
endowed with the weak$^*$ topology. 
The subset 
$\mathcal M_1(\Lambda^{\NN^{[k]}})\subset \mathcal M(\Lambda^{\NN^{[k]}})$ denotes the 
(compact metrizable) space of all probability measures on $\Lambda^{\NN^{[k]}}$.

We denote by $\omega_1$ the first uncountable ordinal
and by $\mathcal M_{\rm c}(\omega_1^{\NN^{[k]}})$ the set of all probability measures on $\omega_1^{\NN^{[k]}}$ 
with compact support, that is, with support in $\alpha^{\NN^{[k]}}$ for some $\alpha<\omega_1$.

We let $\Ic\subset \NN^{\NN}$ be the family of all maps $\sigma: \NN\to \NN$
which are strictly increasing, 
and we let $s:\NN\to\NN$ be the shift map defined as $s(i)=i+1$ for all $i\in\NN$.

As a general rule, given a map $f$ we use the symbol $f_*$ (resp. $f^*$) to indicate a covariant (resp. contravariant) 
map induced by $f$. Due to possible iteration of the previous rule, we still use the symbol $*$ to sum-up a certain number of $*$'s,
keeping the previous convention in order to distinguish a covariant from a contravariant action.
In particular, given a map $\sigma: \NN\to \NN$ we let 
$\sigma_*: \NN^{[k]}\to \NN^{[k]}$ be defined as 
$\sigma_*\left(i_0,\ldots,i_{k-1}\right) = \left(\sigma(i_0),\ldots,\sigma(i_{k-1}) \right)$, 
we let $\sigma^*: \Lambda^{\NN^{[k]}}\to \Lambda^{\NN^{[k]}}$
be defined as $\sigma^*(x)_i=x_{\sigma_*(i)}$ for all $i\in\NN^{[k]}$, 
and we let $\sigma^*:  \mathcal M(\Lambda^{\NN^{[k]}})\to \mathcal M(\Lambda^{\NN^{[k]}})$ 
be the corresponding pushforward map.

%%%%%%%%%%%%%%%%%%%%%%%%%%%%%%%%%%%%%%%%%%%%%%%%%%%%%%%%%%%%%
\section{Contractible measures and graphs}\label{secconm}
%%%%%%%%%%%%%%%%%%%%%%%%%%%%%%%%%%%%%%%%%%%%%%%%%%%%%%%%%%%%%

%%%%%%%%%%%%%%%%%%%%%%%%%%%%%%%%%%%%%%%%%%%%%%%%%%%%%%%%
\subsection{Contractible measures}
%%%%%%%%%%%%%%%%%%%%%%%%%%%%%%%%%%%%%%%%%%%%%%%%%%%%%%%%

Let $k\in\NN$, let $\Lambda$ be a compact metric space, and let $m$ be a Borel probability measure on 
$\Lambda^{\NN^{[k]}}$. We say that $m$ is \emph{(jointly) contractible} if $m=\sigma^* (m)$ for all $\sigma\in\Ic$.

Noice that, when $k=1$, the notion of contractible measure reduces to the well-known notion of 
exchangeable measure, introduced by De Finetti in \cite{DF:30}.

Given $k\in\NN$ and $f:[0,1]^{k^{[\le k]}}\to \Lambda$, we let 
$c: \NN^{[k]}\times k^{[\le k]}\to \NN^{[\le k]}$ be the composition function,
and $\hat f:[0,1]^{\NN^{[\le k]}}\to \Lambda^{\NN^{[k]}}$ 
be defined as the composition of $f_*$ and $c^*$, that is
\[
\hat f= f_*\circ\, c^*,
\]
where $f_*:\big( [0,1]^{k^{[\le k]}}\big)^{\NN^{[k]}}\to \Lambda^{\NN^{[k]}}$, 
$c^*:[0,1]^{\NN^{[\le k]}}\to [0,1]^{\NN^{[k]}\times k^{[\le k]}}$.
Notice that $\hat f$ induces a map $\hat f_*: \mathcal M_1([0,1]^{\NN^{[\le k]}})
\to \mathcal M_1(\Lambda^{\NN^{[k]}})$.

Contractible measures admit the following representation \cite[Theorem 7.15]{K:05}:

\bt \label{teoah} 
Let $\Lambda$ be a compact subset of $[0,1]$ 
and let $m$ be a contractible measure on $\Lambda^{\NN^{[k]}}$.
Then, there exists a measurable function $f:[0,1]^{k^{[\le k]}}\to \Lambda$ such that 
$m=\hat f_*(\mathcal L)$, where $\mathcal L$ is the product Lebesgue measure on $[0,1]^{\NN^{[\le k]}}$.
\et

We observe that Theorem \ref{teoah}, in the particular case $k=1$, implies 
the representation result for exchangeable measures due to De Finetti \cite{DF:30} (see also \cite{HS:55}).

%%%%%%%%%%%%%%%%%%%%%%%%%%%%%%%%%%%%%%%%%%%%%%%%%%%%%%%%
\subsection{Asymptotically contractible measures}
%%%%%%%%%%%%%%%%%%%%%%%%%%%%%%%%%%%%%%%%%%%%%%%%%%%%%%%%

We say that a measure $m\in \mathcal M_1(\Lambda^{\NN^{[k]}})$ is \emph{asymptotically contractible} 
if there exists a contractible measure $m'$ such that there exists the limit
\[
m' = \lim_{\underset{\theta \in \Ic} {\min \theta \to \infty}} \theta^*(m).
\]
%In the probability measures $m_k := (s^*)^k\,\sigma^*(m)$ weakly$^*$ converge to $m'$ as $k\to +\infty$, 
%for all fixed $\sigma\in\Ic$. 
The following result can be proved as in \cite[Theorem B.8]{BMN:12}.

\bprop\label{propasc}
Given $m\in \mathcal M_1(\Lambda^{\NN^{[k]}})$ there exists $\sigma\in\Ic$ such that 
the measure $\sigma^*(m)$ is asymptotically contractible.
\eprop

\begin{proof}
Fix $m\in {\ca M}_1(\Lambda^{\NN^{[k]}})$. Given $r \in \NN$ consider the function 
$f\colon {\NN}^{[r]}\to {\ca M}_1(\Lambda^{r^{[k]}})$ 
sending $\iota\in {\NN}^{[r]}$ to $\iota^*(m)\in {\ca M}_1(\Lambda^{r^{[k]}})$. 
By Lemma \ref{lempie} below, applied with $M={\ca M}_1(\Lambda^{r^{[k]}})$, 
there is an infinite set $J_r \subset \NN$ such that
\begin{equation} \label{limit}
\lim_{\underset{\iota \in J^{[r]}_r}{\min(\iota) \to \infty}} \iota^*(m)
\end{equation}
exists in ${\ca M}^1(\Lambda^{r^{[k]}})$. By a diagonal argument we choose the same set $J = J_r$ for all $r$. 
Letting $\sigma \in \Ic$ be such that $\sigma (\NN) = J$, we claim that $\sigma^*(m)$ is asymptotically exchangeable. 
To this aim consider $m_k := (s^*)^k \sigma^*(m) \in {\ca M}_1(\Lambda^{\NN^{[k]}})$. 
By compactness there exists an accumulation point $m'\in {\ca M}_1(\Lambda^{\NN^{[k]}})$ of 
the sequence $\{m_k\}_{k\in \NN}$, and we claim that
\begin{equation} \label{mprime}
\lim_{\underset{\theta \in \Ic}{\min(\theta) \to \infty}} \theta^*\sigma^*(m) = m'\, .
\end{equation}
%
%hence in particular $m_k \to m'$, by taking $\theta = (s^*)^k$. 
Notice that the claim also implies that $m'$ is contractible. 
Indeed, given an increasing function $\gamma \colon \NN \to \NN$, to show $\gamma^*(m') = m'$ 
it is enough to replace $\theta$ with $\theta \circ \gamma$ in \eqref{mprime}. 
Since the subset of $C(\Lambda^{\NN^{[k]}})$ consising of the functions depending 
on finitely many coordinates is dense in $C(\Lambda^{\NN^{[k]}})$, 
it suffices to prove that the limit
\begin{equation} \label{proj}
\lim_{\underset{\theta \in \Ic}{\min(\theta) \to \infty}} \iota^* \theta^* \sigma^*(m) 
\end{equation}
exists in ${\ca M}_1(\Lambda^{r^{[k]}})$ for all $r\in \NN$ and $\iota \in {\NN}^{[r]}$, but this is just a special case of \eqref{limit}.
\end{proof}

We conclude this section by stating 
a topological version of Ramsey Theorem \cite[Lemma A.1]{BMN:12}.

\bl \label{lempie} Let $M$ be a compact metric space, let $k\in \NN$, 
and let $f:  {\NN}^{[k]}\to M$.  
Then there exists an infinite set $J\subset \NN$ such that there exists the limit
\[
\lim_{\underset{(i_1, \ldots, i_k) \in J^{[k]} } {(i_1, \ldots, i_k) \to +\infty} } f(i_1, \ldots , i_k) .
\]
\el

%%%%%%%%%%%%%%%%%%%%%%%%%%%%%%%%%%%%%%%%%%%%%%%%%%%%%%%%
\subsection{Contractible graphs}\label{seccong}
%%%%%%%%%%%%%%%%%%%%%%%%%%%%%%%%%%%%%%%%%%%%%%%%%%%%%%%%

We now introduce the class of ambient graphs in which we consider Problem \ref{problemv}.
\bd \label{defcg}
Let $G$ be an oriented graph with $V_G=\NN^{[k]}$ for some $k\in\NN$.
We say that $G$ is \emph{contractible} if 
\begin{enumerate}
\item $(a,b)\in E_G\ \Rightarrow\ a_i<b_i$ for all $0\le i\le k-1$;
\item $(a,b)\in E_G\ \Rightarrow\ 
(\sigma_*(a),\sigma_*(b))\in E_G$ for all $\sigma\in\Ic$. 
\end{enumerate}
\ed

We say that a map $\tau:S_\tau\to \NN$ is an 
\emph{order relation} of length $k\in\NN$ if $S_\tau = \{a_0,\ldots,a_{k-1}\}\subset\NN$, with 
$a_0<\cdots <a_{k-1}$, $\tau$ is strictly increasing, and 
$\tau(a_i)>a_i$ for all $0\le i\le k-1$.
We say that two order relations $\tau,\tau^\prime$ are equivalent 
if there exist $\sigma,\sigma^\prime\in\Ic$ such that 
\[
\sigma(a_i) = \sigma^\prime(a^\prime_i)
\quad {\rm and}\quad 
\sigma\circ \tau(a_i) = \sigma^\prime\circ\tau^\prime(a^\prime_i) \quad {\rm for\ all\ }0\le i\le k-1.
\]
We denote by $\E_k$ the set of all equivalence classes of the order relations of length $k$.
Notice that, since the set $\E_k$ is finite, we can assume that all the relations $\tau\in\E_k$ are
defined on the same domain $S\subset\NN$, i.e. $S_\tau=S_{\tau^\prime}=S$ for all $\tau,\tau^\prime\in\E_k$.

We say that the couple $(v,v^\prime)\in \NN^{[k]}\times \NN^{[k]}$ satisfies the order relation $\tau$ 
if there exists $\sigma\in\Ic$ such that $v_i = \sigma(a_i)$ 
and $v^\prime_i=\sigma\circ\tau(a_i)$ for all $0\le i\le k-1$. Notice that $(v,v^\prime)$
and $(v^\prime,v)$ cannot both satisfy the same order relation $\tau$.
Notice also that $\tau$ can be naturally extended to a function 
$\tau_*:S_\tau^{[\le k]}\to\NN^{[\le k]}$ 
defined on all the $r$-tuples of elements in $S_\tau$.
Given an order relation $\tau$, we let 
\begin{equation}\label{sigmatau}
\Sigma_\tau := \{ a_i\in S_\tau:\, i\le k-1,\, \tau(a_i)=a_j\ {\rm for\ some}\ 0\le j\le k-1 \}\subseteq S_\tau\,.
%\\
%\psi_\tau &:=& \tau|_{\Sigma_\tau}:\,\Sigma_\tau\to S,\qquad {\rm so\ that\ }
%\psi_\tau(\Sigma_\tau)=S\cap \tau(S).
\end{equation}

Given $\C\subseteq \E_k$ we let $G_\C$ be the graph with set of vertices $\NN^{[k]}$, and 
such that $(v,v^\prime)$ is an edge of $G_\C$ if and only if $(v,v^\prime)$ satisfies $\tau$
for some $\tau\in\C$.   
Notice that the graph $G_\C$ is contractible for all $\C\subseteq \E_k$.
%On the other hand, any contractible graph $G$ is eventually equal to some $G_\C$, in the sense that
%there exists $n\in\NN$, depending on $G$, such that the graph map $s_*^n(G)=G_\C$, for some $\C\subseteq \E_k$.    

\begin{definition}
We say that a contractible graph $G$ is \emph{simple} if it is eventually equal to some graph $G_\tau$, i.e.
if $s_*^n(G)=G_\tau$, for some $\tau\in \E_k$ and $n\in\NN$.
\end{definition}

A relevant example of such graphs is the graph $G_k=(V_{G_k}, E_{G_k})$ with vertices $V_{G_k}=\NN^{[k]}$ 
and such that $(v,v^\prime)\in E_{G_k}$ iff $v^\prime_i=v_{i+1}$ for all $i\in \{0,\ldots,k-2\}$. 
%Notice that $E_{G_k}$ can be identified with $\NN^{[k+1]}$.

%%%%%%%%%%%%%%%%%%%%%%%%%%%%%%%%%%%%%%%%%%%%%%%%%%%%%%%%%%%%%%%%%%%%%%%%%%%%%%%
\section{Probability thresholds}\label{sectrash}
%%%%%%%%%%%%%%%%%%%%%%%%%%%%%%%%%%%%%%%%%%%%%%%%%%%%%%%%%%%%%%%%%%%%%%%%%%%%%%

In this section we give an answer to Problem \ref{problemv} under the assumption that $G$ is a simple contractible graph. 

%%%%%%%%%%%%%%%%%%%%%%%%%%%%%%%%%%%%%%%%%%%%%%%%%%%%%%%%%%%%%%%%%%%%%%%%%%%%%%%%%%%%%
\subsection{A variational problem in the unit cube}\label{secvar}
%%%%%%%%%%%%%%%%%%%%%%%%%%%%%%%%%%%%%%%%%%%%%%%%%%%%%%%%%%%%%%%%%%%%%%%%%%%%%%%%%%%%%

We first discuss a variational problem on the unit cube, which is related to Problem \ref{problemv}.
{}From now on, given a set $A\subseteq [0,1]^S$ where $S$ is a countable set, we set
$|A| := \mathcal L(A)$
where $\mathcal L$ is the product Lebesgue measure on $[0,1]^S$.

\begin{lemma}\label{lemvar}
Let $S$ be a set, let $S_0\subset S$ be a finite subset, 
and let $\tau:S_0\to S$, so that $\tau^*$ maps $[0,1]^S$ into $[0,1]^{S_0}$.
For all measurable functions $f:[0,1]^{S_0}\to [0,1]$ we let 
\begin{equation}\label{eqvar}
Z_\tau(f):= \left\{ x\in [0,1]^{S}:\ f\left( x|_{S_0}\right)>f \left( \tau^*(x)\right)\right\}\,.
\end{equation}
If $\tau$ has no cycles, we have
\begin{equation}\label{eqvarbis}
\sup_{f:[0,1]^{S_0}\to [0,1]}\vert Z_\tau(f)\vert  = 1-\frac{1}{w(\tau)}\,,
\end{equation}
where 
\[
w(\tau) := 1 + \max \{ n\in \NN: \exists\, s\in S_0\ {\rm s.\,t.\,}
\tau^j(s)\in S_0\ \forall\ 1\le j\le n-1\}
\ge 2.
\]
\end{lemma}

\begin{proof}
Assume first that $\tau$ can be extended to a $p$-periodic function $\tilde\tau$ on the whole of $S$, for some $p\in\NN$.
Let $f:[0,1]^{S_0}\to [0,1]$ and notice that 
\[
\bigcap_{j=0}^{p-1}(\tilde\tau^*)^j(Z_\tau(f)) = \emptyset.
\]
It then follows
\[
p \left(1-|Z_\tau(f)|\right) = p |Z_\tau(f)^c| \ge  
\left|\bigcup_{j=0}^{p-1}\left((\tilde\tau^{*})^j(Z_\tau(f))\right)^c\right| = 
1- \left|\bigcap_{j=0}^{p-1}({\tilde\tau^*})^j(Z_\tau(f))\right| =1,
\]
which implies 
\begin{equation}\label{eqeqeq}
\vert Z_\tau(f)\vert  \le 1-\frac{1}{p}\,.
\end{equation}
Observe that, by definition of $w(\tau)$, 
we can find a $w(\tau)$-periodic function $\tilde\tau:S\to S$ such that 
$\tilde\tau|_{S_0}=\tau|_{S_0}$. From \eqref{eqeqeq} it then follows
\[
\vert Z_\tau(f)\vert  \le 1-\frac{1}{w(\tau)}\,.
\]

\smallskip 

We now prove the opposite inequality. By definition of $w(\tau)$ there exists 
$\overline s\in S_0$ such that $\tau^j(\overline s)\in S_0$ for all $1\le j\le w(\tau)-2$. 
Let us consider the function 
\[
f_\eps(x)= \frac{1}{1+\eps w(\tau)}\left(\, \max_{0\le j\le w(\tau)-2}x_{\tau^j(\overline s)} + \eps\bar j\right),
\]
where $1/(1+\eps w(\tau))$ is just a normalization factor in order to have $f_\eps(x)\in [0,1]$, and
$\bar j=\bar j(x)\in \NN$ is such that $x_{\tau^{\bar j}(\overline s)}=\max^{}_{0\le j\le w(\tau)-2}x_{\tau^j(\overline s)}$.
Notice that $\bar j(x)$ is well-defined out of a set of zero measure. 
We then have
\[
\vert Z_\tau(f_\eps)\vert \ge \lim_{\eps\to 0}
\vert Z_\tau(f_\eps)\vert  = 1-\frac{1}{w(\tau)}
\]
which gives \eqref{eqvarbis}.
\end{proof}

Notice that $w(\tau)$ depends only only on $(\tau,S_0)$ and is independent of $S$, so that for instance 
one can substitute $S$ with the finite set $S_0\cup \tau(S_0)$.

\bexa\label{exshift} 
In the particular case where $S=\NN$, $S_0=\{ 1,\ldots,k\}$ 
and $\tau$ is the shift map restricted to $S_0$, i.e. $\tau(i)= i+1$ 
we have $w(\tau)=k+1$, so that 
\[
\sup_{f:[0,1]^{S_0}\to I}\left\vert Z_\tau(f)\right\vert = 1-\frac{1}{k+1}\,.
\]
\eexa

%%%%%%%%%%%%%%%%%%%%%%%%%%%%%%%%%%%%%%%%%%%%%%%%%%%%%%%%%%%%%%%%
\subsection{A canonical probability space}\label{secphi}
%%%%%%%%%%%%%%%%%%%%%%%%%%%%%%%%%%%%%%%%%%%%%%%%%%%%%%%%%%%%%%%%%

Following \cite{BMN:12}, we reformulate Problem \ref{problemv} as a variational problem on a suitable space of sequences.
For all $\omega\in \Om$, we consider the oriented graph $X(\omega)<G$ whose vertices 
are all $v\in V_G$ such that $\omega\in X_{v}$.  
Given $v\in \NN^{[k]}$, we also let $Y_v\subseteq X_v$ be the subset of all $\omega$ such that 
$X(\omega)$ contains an infinite path starting from $v$, i.e. there exists an infinite 
sequence $\{v_k\}_{k\in\NN}$, with $v_1=v$ and $\omega\in \bigcap_k X_{v_k}$. 

We recall that a partially ordered set admits a decreasing function into the first uncountable ordinal $\omega_1$ 
(called height or rank function) if and only if it has no infinite increasing sequences,
we can define a measurable map $\phi: \Omega\times\NN^{[k]}\to \omega_1+1$ by setting 
\[
\phi(\omega,v) = \left\{ \begin{array}{ll}
\displaystyle\sup_{\{v'\in \NN^{[k]}:\, (v,v')\in E_G,\,\omega\in X_{v'}\}} \phi(\omega,v')+1
& {\rm if\ }\omega\in X_v\setminus Y_v,
\\
0 & {\rm if\ }\omega\not\in X_v
\\
\omega_1
& {\rm if\ }\omega\in Y_v.
\end{array}\right. 
\]
For simplicity of notation, we still denote by $\phi$ the map $\phi: \Om\to (\omega_1+1)^{\NN^{[k]}}$ 
defined as $\phi(\omega)_v = \phi(\omega,v)$.
 
\br  
We have $\phi(\omega,v) = \phi_{\omega^{}_1}(\omega,v)$ where $\phi_\alpha \colon \Om \to (\omega_1+1)^{\NN^{[k]}}$ is   
the truncation $\phi_\alpha:=\min (\phi, \alpha)$, that we can equivalently define by induction on $\alpha \leq \omega_1$ as  
\[
\begin{array}{lll} \label{defrec} 
\phi_0(\omega,v) &=& 0 \\
\phi_\alpha(\omega,v) &=& \sup \,\{ \phi_{\beta}(\omega,v')+1 :\, \beta<\alpha,\, (v,v')\in E_G,\,\omega\in X_{v'}\}.
\end{array}
\]
\er

Notice that $\phi(\omega,v)<\omega_1$ iff there is no infinite path 
in $X(\omega)$ starting from $v$, that is $\omega\not\in Y_v$, and in this case $\phi(\omega,v)$ 
is precisely the height of the vertex $v$ in $X(\omega)$. 
Indeed, if $\omega\in Y_v$ then 
$\phi(\omega,v)=\omega_1$ by definition. On the other hand, if $\phi(\omega,v)=\omega_1$ then there exists $v'$ such that 
$(v,v')\in E_G$, $\omega\in X_{v'}$ and $\phi(\omega,v')=\omega_1$, since otherwise $\phi(\omega,v)$ would also be less than $\omega_1$,
being the supremum of a countable set of countable ordinals. By iteration we then get $\omega\in Y_v$.
In particular, if $X$ has no infinite paths, then the function 
$\phi$ takes value in $\omega_1^{\NN^{[k]}}$ and, if there are no paths of length $p$, 
then it takes values in $p^{\NN^{[k]}}$. 
%It also follows that 
%$F$ has an infinite path if and only if the measurable set $\{x:\,\tilde\phi(x)=\omega_1\}$ is non-empty.

%As in \cite[Lemma 4.3]{BMN:08}, if $F$ has no infinite paths one can show that the map $\phi$ is essentially bounded, 
%even if not necessarily bounded everywhere.

\bl\label{lemnev}
The following assertions hold:
\begin{enumerate}
\item the set $P:= \{\omega \in \Om :\, X(\omega) \text{ has an infinite path } \}$ is $\mu$-measurable; 
\item for all $\alpha \leq \omega_1$ and $v\in V_G$, the set $\{\omega \in \Om  :\, \phi(\omega,v) = \alpha\}$ is $\mu$-measurable; 
\item $\phi\colon \Omega \to {(\omega_1+1)}^{V_G}$ is $\mu$-measurable and its  restriction to $\Omega\setminus P$ is essentially bounded. 
%namely for some $\alpha_0<\omega_1$ it takes values in $\alpha_0^{V_G}$ outside of a $\mu$-null set. 
\end{enumerate} 
%If $F$ has no infinite paths, then $\tilde\phi\in L^\infty(X,\mu)$. 
\el

\bp
Since taking the supremum over a countable set preserves measurability, 
it follows that the sets $\{\omega \in \Om : \phi(\omega,v)=\alpha\}$ are measurable
for all $v\in V_G$ and $\alpha < \omega_1$. We will show that 
$\{\omega \in \Om : \phi(\omega,v)=\omega_1\}$ is $\mu$-measurable, namely it is the union of a measurable set 
and a $\mu$-null set. Fix $v\in V_G$. The sequence of values
$\mu\left(\left\{ \omega \in \Om : \phi(\omega,v)\le\beta\right\}\right)$ is increasing with respect to the 
countable ordinal $\beta$ and uniformly bounded by $1=\mu(\Omega),$ therefore it is stationary at some finite value. 
So there is $\alpha_0<\omega_1$ such that 
\[
\mu\left(\left\{ \omega \in \Om :\ \phi(\omega,v)=\beta\right\}\right)=0 
\qquad \mbox{ for } \alpha_0\le\beta<\omega_1\,. 
\]
It follows that $\{\omega \in \Om : \phi(\omega,v)=\omega_1\}$ is $\mu$-measurable and $\phi$ is $\mu$-measurable. 
Since $P = \cup_v \{\omega \in \Om : \phi(\omega,v) = \omega_1\}$, we have that $P$ is $\mu$-measurable, too. 
\ep

As a consequence, if $X$ has no infinite paths, then 
the function $\phi$ maps $\Om$ into the Cantor space 
$\alpha^{\NN^{[k]}}$ for some $\alpha<\omega_1$ (up to a set of zero measure),
so that it induces a probability measure $m=\phi_*(\mu)$ on $\omega_1^{\NN^{[k]}}$ concentrated on $\alpha^{\NN^{[k]}}$,
i.e. $m(\alpha^{\NN^{[k]}})=1$. Moreover, for all $v\in V_G$ we have
\begin{equation}\label{av}
\phi(X_v)\subseteq 
A_v:=\left\{ x\in \omega_1^{\NN^{[k]}}:\, x_v>x_{v'}\, {\rm for\ all}\, v'\, {\rm such\, that}\,(v,v')\in E_G\right\}.
\end{equation}
It then follows 
\begin{equation}\label{soglia}
\lambda_G = \sup_{m\in \mathcal M_{\rm c}(\omega_1^{\NN^{[k]}})}\inf_{v\in \NN^{[k]}}m(A_{v}).
\end{equation}

\subsection{Infinite paths.} We are ready to state the main result of this paper.

\bt \label{teoromitov}
Let $G$ be a simple contractible graph with $V_G=\NN^{[k]}$, i.e. 
$G=G_\tau$ for some $\tau\in\E_k$, up to a finite number of shifts. 
Then
\[
\lambda_G = 
%\sup_{f:[0,1]^{\Sigma_\tau^{[\le k]}}\to [0,1]}
%\left\vert \left\{ x\in [0,1]^{\NN^{[\le k]}}:\ f(x|_{\Sigma_\tau^{[\le k]}})>f(x|_{\tau(\Sigma_\tau)^{[\le k]}})\right\}\right\vert
%\\ &=& 
1 - \frac{1}{w\left(\tau|_{\Sigma_\tau}\right)}\,.
\]
%Then
%\begin{equation}\label{mmv}
%\sup_{m\in \mathcal M_{\rm c}(\omega_1^{\NN^{[k]}})}\inf_{v\in \NN^{[k]}}m(A_{v})= \lambda_G \,.
%\end{equation}
%where $\C\subset \E_k$ is such that $G=G_\C$ up to a finite number of shifts.
In particular, letting $X: \Om\to 2^{V_G}$, 
the random graph $X$ has an infinite path if
\begin{equation}\label{hhv}
\lambda:=
\inf_{v\in V_G}\mu(X_{v})> \lambda_G\,.
\end{equation}
On the contrary, if $\lambda<\lambda_G$ we can find $X$ such that $X(\omega)$ has no infinite paths 
for some $\omega\in \Om$.
\et

\bp
We divide the proof into several steps.

\smallskip

{\it Step 1.} Recalling \eqref{soglia}, we want to compute the supremum
\[
\sup_{m\in \mathcal M_{\rm c}(\omega_1^{\NN^{[k]}})}\inf_{v\in \NN^{[k]}}m(A_{v}).
\]
Since the support of $m$ is contained in $\alpha_0^{\NN^{[k]}}$, for some compact countable ordinal $\alpha_0$,
thanks to Proposition \ref{propasc} we can assume that $m$ is asymptotically contractible,
so that in particular the sequence $m_k=(s^*)^k(m)$ converges to a contractible measure $m^\prime\in \mathcal M_1(\alpha_0^{\NN^{[k]}})$ 
in the weak$^*$ topology of $\mathcal M_1(\alpha_0^{\NN^{[k]}})$. Moreover, 
one can prove by ordinal induction as in \cite[Lemma 4.4]{BMN:12} that for all $\alpha<\omega_1$ there holds 
\begin{equation*}
\inf_{v\in \NN^{[k]}}m\left( \left\{ x\in A_v: x_v\le \alpha\right\}\right)\le 
\inf_{v\in \NN^{[k]}}m^\prime\left( \left\{ x\in A_v: x_v\le\alpha\right\}\right).
\end{equation*}
Therefore, we can assume that the measure $m$ in \eqref{soglia} is contractible.

\smallskip

{\it Step 2.} Recalling the definition of the set $A_v$ in \eqref{av} and
letting $S=S_\tau$ be the domain of $\tau$,
by Theorem \ref{teoah} we have 
\begin{eqnarray}\label{estimo}
&& \!\!\!\sup_{m\in \mathcal M_{\rm c}(\omega_1^{\NN^{[k]}})}\inf_{v\in \NN^{[k]}}m(A_{v}) = 
\sup_{\underset{m\,\rm contractible}{m\in \mathcal M_{\rm c}(\omega_1^{\NN^{[k]}})}}\ \inf_{v\in \NN^{[k]}}m\left(  A_v\right)
= \sup_{\underset{m\,\rm contractible}{m\in \mathcal M_{\rm c}(\omega_1^{\NN^{[k]}})}}\ m\left(  A_S\right) 
\\ \nonumber
&& = \sup_{\underset{\alpha_0<\omega_1}{f:[0,1]^{S^{[\le k]}}\!\!\to \alpha_0}}
\left\vert 
\left\{ x\in [0,1]^{\NN^{[\le k]}}: f(x|_{S^{[\le k]}})>
f(\tau^*\sigma^*(x))
\ \forall\sigma\in\Ic \,{\rm s.\,t.\,} \sigma|_S={\rm id}\right\}\right\vert
\\ \nonumber
&& \le \sup_{f:[0,1]^{S^{[\le k]}}\!\!\to [0,1]}
\!\!\left\vert 
\left\{ x\in [0,1]^{\NN^{[\le k]}}: f(x|_{S^{[\le k]}})>
f(\tau^*\sigma^*(x))
\ \forall\sigma\in\Ic \,{\rm s.\,t.\,} \sigma|_S={\rm id}\right\}\right\vert
\end{eqnarray}
where the last inequality follows since every countable ordinal $\alpha_0$ can be embedded in $[0,1]$ 
with the induced ordering. We will show below that this inequality is indeed an equality.

\smallskip

{\it Step 3.} Recalling the definition of $\Sigma_\tau$ in \eqref{sigmatau}, we
observe that every $x\in [0,1]^{\NN^{[\le k]}}$ can be uniquely written as $x=(y,z)$, 
with $y=x|_{S^{[\le k]}}$ and $z=x|_{\NN^{[\le k]}\setminus S^{[\le k]}}$. Similarly,
we sometimes write $y\in [0,1]^{S^{[\le k]}}$ as $y=(y',t)$,
with $y'=y|_{\Sigma_\tau^{[\le k]}}$ and $z=x|_{S^{[\le k]}\setminus \Sigma_\tau^{[\le k]}}$.
For simplicity of notation, we let in the following $\tau_0:\Sigma_\tau\to S$ and 
$\tau_1:S\setminus\Sigma_\tau\to \NN\setminus S$ be the restrictions of $\tau$ 
to $\Sigma_\tau$ and $S\setminus\Sigma_\tau$, respectively. We now show that 
\begin{eqnarray}\label{lety}
&&\left\vert 
\left\{ x\in [0,1]^{\NN^{[\le k]}}: f(x|_{S^{[\le k]}})>
f(\tau^*\sigma^*(x))
\ \forall\sigma\in\Ic \,{\rm s.\,t.\,} \sigma|_S={\rm id}\right\}\right\vert
\\ \nonumber 
&&=\left\vert \left\{ y\in [0,1]^{S^{[\le k]}}:\, f(y)>f\left(\tau_0^*(y),t\right)\ \forae\, 
t\in [0,1]^{S^{[\le k]}\setminus \Sigma_\tau^{[\le k]}}
\right\}\right\vert
\end{eqnarray}
where $\tau|_{\Sigma_\tau}:\Sigma_\tau\to S$, so that $\tau_0^*: [0,1]^{S^{[\le k]}}\to [0,1]^{\Sigma_\tau^{[\le k]}}$.
Indeed, for all $y\in S^{[\le k]}$ we set
\[
J(y) = \left\{ t\in [0,1]^{S^{[\le k]}\setminus \Sigma_\tau^{[\le k]}}:\,
f(y)>f\left(\tau_0^*(y),t\right)
\right\}\,.
\]
Then, letting
\[
\si = \left\{ \sigma\in\Ic :\,\sigma|_S={\rm id}\right\}\,,
\]
by Fubini-Tonelli Theorem we have
\begin{eqnarray}\label{nonni}
\nonumber
&&\left\vert 
\left\{ x\in [0,1]^{\NN^{[\le k]}}: f(x|_{S^{[\le k]}})>
f(\tau^*\sigma^*(x))
\ \forall\sigma\in\si\right\}\right\vert 
\\ \nonumber
&&= \left\vert 
\left\{ x=(y,z)\in [0,1]^{\NN^{[\le k]}}: f(y)>
f\left(\tau_0^*(y),\tau_1^*\sigma_1^*(z)\right)\right.\right.
\forall\sigma\in\si\Big\}\Big\vert
\\ \nonumber
&&= \int_{[0,1]^{S^{[\le k]}}} \left\vert \left\{ z\in [0,1]^{\NN^{[\le k]}\setminus S^{[\le k]}}:
f(y)>f\left(\tau_0^*(y),\tau_1^*\sigma_1^*(z)\right) \forall\sigma\in\si
\right\}\right\vert\,dy
\\
&&= \int_{[0,1]^{S^{[\le k]}}} \left\vert \left\{ z\in [0,1]^{\NN^{[\le k]}\setminus S^{[\le k]}}:
\tau_1^*\sigma_1^*(z)\in J(y)\, \forall\sigma\in\si\right\}\right\vert\,dy
\end{eqnarray}
where we denoted by $\sigma_1:\NN\setminus S\to \NN\setminus S$ the restriction of $\sigma$ to $\NN\setminus S$.
Notice that the set
\[
\left\{ z\in [0,1]^{\NN^{[\le k]}\setminus S^{[\le k]}}:
\tau_1^*\sigma_1^*(z)\in J(y)\, \forall\sigma\in\si\right\},
\]
appearing in the last term of \eqref{nonni}, is contained in an infinite product of the set $J(y)$,
therefore its measure is nonzero if and only if $|J(y)|=1$. From \eqref{nonni} it then follows
\begin{eqnarray*}
&&\left\vert 
\left\{ x\in [0,1]^{\NN^{[\le k]}}: f(x|_{S^{[\le k]}})>
f(\tau^*\sigma^*(x))
\ \forall\sigma\in\si\right\}\right\vert 
\\
&&= \int_{[0,1]^{S^{[\le k]}}}\big\lfloor\vert J(y)\vert\big\rfloor\,dy
\\
&&= \left\vert \left\{ y\in [0,1]^{S^{[\le k]}}:\, \vert J(y)\vert=1\right\}\right\vert
\\
&&= \left\vert \left\{ y\in [0,1]^{S^{[\le k]}}:\, f(y)>f(\tau_0^*(y),t)\ \forae\, 
t\in [0,1]^{S^{[\le k]}\setminus \Sigma_\tau^{[\le k]}}
\right\}\right\vert
\end{eqnarray*}
which proves \eqref{lety}.

\smallskip

{\it Step 4.} Given $f:[0,1]^{S^{[\le k]}}\to [0,1]$ and $\eps>0$, we define
$g_{f,\eps}: \Sigma_\tau^{[\le k]}\to [0,1]$ as
$$
g_{f,\eps}(y') = \sup \left\{\lambda\in [0,1]:\,
\left\vert\{t\in[0,1]^{S^{[\le k]}\setminus \Sigma_\tau^{[\le k]}}: f(y',t)>\lambda\}\right\vert\ge \eps\right\}.
$$
For all $y\in [0,1]^{S^{[\le k]}}$ we also let
$f_\eps(y) = \min\big( f(y),\ g_{f,\eps}(\tau_0^*(y))\big)$.
%\[
%g_{f,\eps}(\tau_0^*(y)) =\ {\rm ess}\!\!\!\!\!\!\!\!\!\!\!\!\!\!\sup_{t\in [0,1]^{S^{[\le k]}\setminus \Sigma_\tau^{[\le k]}}} f_\eps(\tau_0^*(y),t).
%\]
Notice that from the definition of $g_{f,\eps}$ it follows that the set 
\[
A_\eps = \left\{ y\in [0,1]^{S^{[\le k]}}:\, f(y)>g_{f,\eps}(\tau_0^*(y))\right\}
\]
has measure $|A_\eps|\le \eps$, moreover $f=f_\eps$ on $[0,1]^{S^{[\le k]}}\setminus A_\eps$.
Notice also that the function $t\to f(\tau_0^*(y),t)$
has essential supremum $g_{f,\eps}(\tau_0^*(y))$,
and attains such value on a set of measure bounded below by $\eps$.
We then get
\begin{eqnarray}\label{eqine}
\nonumber 
&&\left\vert \left\{ y\in [0,1]^{S^{[\le k]}}:\, f(y)>f(\tau_0^*(y),t)\ \forae\, 
t\in [0,1]^{S^{[\le k]}\setminus \Sigma_\tau^{[\le k]}}
\right\}\right\vert 
\\ \nonumber
&& \le
\ \left\vert \left\{ y\in [0,1]^{S^{[\le k]}}\setminus A_\eps:\, f_\eps(y)>f_\eps(\tau_0^*(y),t)\ \forae\, 
t\in [0,1]^{S^{[\le k]}\setminus \Sigma_\tau^{[\le k]}}
\right\}\right\vert + \eps
\\
&& \le
\ \left\vert \left\{ y\in [0,1]^{S^{[\le k]}}: \, f_\eps(y)>f_\eps(\tau_0^*(y),t)\ \forae\, 
t\in [0,1]^{S^{[\le k]}\setminus \Sigma_\tau^{[\le k]}}
\right\}\right\vert +\eps
\\ \nonumber
&& \le 
\ \Big\vert \Big\{ y\in [0,1]^{S^{[\le k]}}:\, f_\eps(y)>
\,{\rm ess}\!\!\!\!\!\!\!\!\!\!\!\!\!\!\sup_{t\in [0,1]^{S^{[\le k]}\setminus \Sigma_\tau^{[\le k]}}} f_\eps(\tau_0^*(y),t)
\Big\}\Big\vert +\eps
\\ \nonumber
&& \le 
\ \left\vert \left\{ y\in [0,1]^{S^{[\le k]}}:\, g_{f,\eps}(y|_{\Sigma_\tau^{[\le k]}})>g_{f,\eps}(\tau_0^*(y))\right\}\right\vert + \eps\,.
\end{eqnarray}
Notice that \eqref{eqine} implies 
\begin{eqnarray}\label{lupy}
\nonumber
&& \sup_{f:[0,1]^{S^{[\le k]}}\to [0,1]} 
\left\vert \left\{ y\in [0,1]^{S^{[\le k]}}:\, f(y)>f(\tau_0^*(y),t)\ \forae\, 
t\in [0,1]^{S^{[\le k]}\setminus \Sigma_\tau^{[\le k]}}
\right\}\right\vert  
\\
&& = \sup_{g:[0,1]^{\Sigma_\tau^{[\le k]}}\to [0,1]}\left\vert \left\{ 
y\in [0,1]^{S^{[\le k]}}:\, g(y|_{\Sigma_\tau^{[\le k]}})>g(\tau_0^*(y))
\right\}\right\vert
\end{eqnarray}
as the $\ge$ inequality follows immediately by considering functions $f$ depending only 
on the variables in $[0,1]^{\Sigma_\tau^{[\le k]}}$.
Putting together \eqref{lety} and \eqref{lupy} we then get
\begin{eqnarray}\label{lory}
\nonumber 
&&\sup_{f:[0,1]^{S^{[\le k]}}\to [0,1]}
\left\vert 
\left\{ x\in [0,1]^{\NN^{[\le k]}}: f(x|_{S^{[\le k]}})>
f(\tau^*\sigma^*(x))
\ \forall\sigma\in\si\right\}\right\vert
\\ 
&& \sup_{f:[0,1]^{S^{[\le k]}}\to [0,1]} 
\left\vert \left\{ y\in [0,1]^{S^{[\le k]}}:\, f(y)>f(\tau_0^*(y),t)\ \forae\, 
t\in [0,1]^{S^{[\le k]}\setminus \Sigma_\tau^{[\le k]}}
\right\}\right\vert  
\\ \nonumber 
&& = \sup_{g:[0,1]^{\Sigma_\tau^{[\le k]}}\to [0,1]}
\left\vert \left\{ 
y\in [0,1]^{S^{[\le k]}}:\, g(y|_{\Sigma_\tau^{[\le k]}})>g(\tau_0^*(y))
\right\}\right\vert.
\end{eqnarray}

\smallskip

{\it Step 5.} Notice that, if $g_n\to g$ almost everywhere on $[0,1]^{\Sigma_\tau^{[\le k]}}$ as $n\to +\infty$, then 
\begin{eqnarray*}
&& \liminf_{n\to +\infty}\ 
\left\vert \left\{ 
y\in [0,1]^{S^{[\le k]}}:\, g_{n}(y|_{\Sigma_\tau^{[\le k]}})>g_{n}(\tau_0^*(y))
\right\}\right\vert
\\
&& \qquad \ge 
\left\vert \left\{ 
y\in [0,1]^{S^{[\le k]}}:\, g(y|_{\Sigma_\tau^{[\le k]}})>g(\tau_0^*(y))
\right\}\right\vert.
\end{eqnarray*}
This implies that the supremum in \eqref{lory} can be taken over a dense class of $g$'s,
or equivalently of $f$'s, so that in particular the inequality in \eqref{estimo} is indeed an equality. 

{}From \eqref{estimo} with the equality and from \eqref{lory} we then have
\[
\sup_{m\in \mathcal M_{\rm c}(\omega_1^{\NN^{[k]}})}\inf_{v\in \NN^{[k]}}m(A_{v}) 
= \sup_{g:[0,1]^{\Sigma_\tau^{[\le k]}}\to [0,1]}\left\vert \left\{ 
y\in [0,1]^{S^{[\le k]}}:\, g(y|_{\Sigma_\tau^{[\le k]}})>g(\tau_0^*(y))\right\}\right\vert
\]
%%
%\begin{eqnarray}\label{eqfg}
%&& \sup_{f:[0,1]^{S^{[\le k]}}\to [0,1]}
%\left\vert \left\{ x: f(x|_{\sigma(S)^{[\le k]}})>
%f(x|_{\sigma\circ \tau(S)^{[\le k]}})
%\ \forall\sigma\in\Ic\right\}\right\vert
%\\ \nonumber
%&& = \sup_{f:[0,1]^{S^{[\le k]}}\to [0,1]}
%\left\vert 
%\left\{ x: f(x|_{S^{[\le k]}})>
%f(x|_{\sigma\circ \tau(S)^{[\le k]}})
%\ \forall\sigma\in\Ic \,{\rm s.\,t.\,} \sigma|_S={\rm id}\right\}\right\vert
%\\ \nonumber
%&& = \sup_{g:[0,1]^{\Sigma_\tau^{[\le k]}}\to [0,1]}
%\left\vert \left\{ x:\, g(x|_{\Sigma_\tau^{[\le k]}})>g(x|_{\tau(\Sigma_\tau)^{[\le k]}})\right\}\right\vert 
%\end{eqnarray}
%where the last equality follows by letting 
%$$
%g(y) = \sup_{z\in [0,1]^{S^{[\le k]}\setminus \Sigma_\tau^{[\le k]}}} f(y,z) 
%\qquad \forall\, y\in [0,1]^{\Sigma_\tau^{[\le k]}}.
%$$
%Indeed, notice that
%$$
%g(x|_{\Sigma_\tau^{[\le k]}})> f(x|_{S^{[\le k]}})\qquad \forall x\in [0,1]^{\NN^{[\le k]}}
%$$
%for all $f$ which do not attain their maximum on $[0,1]^{S^{[\le k]}\setminus\Sigma_\tau^{[\le k]}}$,
%and 
%$$
%\sup_\sigma f(x|_{\sigma\circ \tau(S)^{[\le k]}}) = g(x|_{\tau(\Sigma_\tau)^{[\le k]}})
%$$
%for almost every $x\in [0,1]^{\NN^{[\le k]}}$.
%It then follows
%$$
%g(x|_{\Sigma_\tau^{[\le k]}})> f(x|_{S^{[\le k]}})\ge
%\sup_\sigma f(x|_{\sigma\circ \tau(S)^{[\le k]}}) = g(x|_{\tau(\Sigma_\tau)^{[\le k]}})
%$$
%which gives \eqref{eqfg}.
and by Lemma \ref{lemvar} we finally get
\[
\sup_{m\in \mathcal M_{\rm c}(\omega_1^{\NN^{[k]}})}\inf_{v\in \NN^{[k]}}m(A_{v}) = 
1 - \frac{1}{w\left(\tau_0\right)}
= 1 - \frac{1}{w\left(\tau|_{\Sigma_\tau}\right)} = \lambda_G\,.
\]

\smallskip

{\it Step 6.} To prove that last assertion it is enough to 
fix $\eps>0$ and consider a contractible measure $m\in \mathcal M_{\rm c}(\omega_1^{\NN^{[k]}})$
such that 
\[
\inf_{v\in \NN^{[k]}}m\left(  A_v\right) \ge \lambda_G -\eps\,.
\]
The conclusion then follows by choosing $\Om=\omega_1^{\NN^{[k]}}$ as probability space, equipped with the measure $m$, 
and letting $X_{v}=A_{v}$. In this way, we define a random graph $X$ with no infinite paths
and such that $m(A_{v})\ge \lambda_G -\eps$ for all $v\in\NN^{[k]}$.
\ep

When $G$ is only contractible, reasoning as above we get the following weaker version of Theorem \ref{teoromitov}.
We recall that, given $\C\subset\E_k$, we may assume that all the maps $\tau\in\C$ are defined on the same domain $S=S_\tau$.
\bprop \label{teoromitogen}
Let $G=G_\C$ be a contractible graph with $V_G=\NN^{[k]}$ and $\C\subset\E_k$. Then 
\[
\lambda_G = \sup_{f:[0,1]^{S^{[\le k]}}\to [0,1]}
\left\vert \left\{ x\in [0,1]^{\NN^{[\le k]}}:\, 
f(x|_{S^{[\le k]}})>f(\tau^*\sigma^*(x))
\ \forall (\sigma,\tau)\in\Ic\times\C\right\}\right\vert.
\]
%\\
%&&\!\!\!\!\!\!\!\!\!\! = \!\!\!\!\!\!\sup_{g:[0,1]^{\Sigma_\tau^{[\le k]}}\to [0,1]}
%\!\left\vert \left\{ x\in [0,1]^{\NN^{[\le k]}}:\ g(x|_{\Sigma_\tau^{[\le k]}})>g(x|_{\tau(\Sigma_\tau)^{[\le k]}})
%\, \forall\tau\in\C\right\}\right\vert.
%
%Then
%\begin{equation*}%\label{mmvgen}
%\sup_{m\in \mathcal M_{\rm c}(\omega_1^{\NN^{[k]}})}\inf_{v\in \NN^{[k]}}m(A_{v})= \lambda_G \,.
%\end{equation*}
In particular, letting $X: \Om\to 2^{V_G}$, 
the random graph $X$ has an infinite path if
\begin{equation*}
\lambda:=
\inf_{v\in V_G}\mu(X_{v})> \lambda_G\,.
\end{equation*}
On the contrary, if $\lambda<\lambda_G$ we can find $X$ such that $X(\omega)$ has no infinite paths 
for some $\omega\in \Om$.
%In particular, letting $F: X\to 2^{V_G}$, 
%the random graph $F<G$ has an infinite path if
%\begin{equation*}%\label{hhvgen}
%\lambda:=
%\inf_{v\in V_G}\mu(X_{v})> \lambda_G\,.
%\end{equation*}
%On the other hand, if $\lambda<\lambda_G$ we can find $F$ such that $F(x)$ has no infinite paths 
%for some $x\in X$.
\eprop

The next result provides a sharp lower bound on the probability of having an infinite path.

\bc\label{corperv}
Assume that the sets $X_{v}$ are such that $\inf_{v\in V_G} \mu(X_{v})=\lambda\ge \lambda_G$.
Let $P_\lambda$ be the set of all $\omega\in \Om$ such that $X(\omega)$ contains an infinite path. 
Then 
\[
\mu(P_\lambda) \ge \frac{\lambda -\lambda_G}{1-\lambda_G}.
\]
\ec

\bp 
Consider the conditional probability $\mu(\cdot \mid \Om \setminus P_\lambda) \in \mathcal M_1(\Om)$. 
For all $v\in V_G$ we have 
\begin{eqnarray}
\mu(X_v \mid \Om\setminus P_\lambda) & \geq & \frac {\mu(X_v) - \mu(P_\lambda)} {1 - \mu(P_\lambda)} \\
\nonumber & \geq & \frac {\lambda - \mu(P_\lambda)} {1 - \mu(P_\lambda)}\,.
\end{eqnarray}
Applying Theorem \ref{teoromitov} or Proposition \ref{teoromitogen} to $\mu(\cdot \mid \Om \setminus P_\lambda)$ it follows that 
$\frac {\lambda - \mu(P_\lambda)} {1 - \mu(P_\lambda)} \leq \lambda_G$, or equivalently $\mu(P_\lambda) \geq \frac {\lambda - \lambda_G} {1 - \lambda_G}$. 
\ep

\begin{remark}\label{remshift}\rm
In the particular case of the graph $G_k$ defined in Section \ref{seccong},  
we have $S_\tau=\{ 1,\ldots, k\}$, $\tau = s|_{S_\tau}$ and $\Sigma_\tau=\{ 1,\ldots, k-1\}$ 
($\Sigma_\tau=\emptyset$ if $k=1$). From Lemma \ref{lemvar} 
and Example \ref{exshift} it then follows that the threshold $\lambda_{G_k}$ is given by
\[
\lambda_{G_k} = 1 - \frac{1}{k}.
\]
Moreover, by Theorem \ref{teoromitov}, $X<G_k$ contains an infinite path if $\lambda > \lambda_{G_k}$. 

It is not clear from this analysis what happens when $\lambda=\lambda_{G_k}$, even if we 
expect that there are still infinite paths (this is proved in \cite{FT:85} and \cite{BMN:12} when $k=2$). 
\end{remark}

%%%%%%%%%%%%%%%%%%%%%%%%%%%%%%%%%%%%%%%%%%%%%%%%%%%%%%%%%%%%%%%%%%%
\subsection{Paths of finite length}\label{secopen}
%%%%%%%%%%%%%%%%%%%%%%%%%%%%%%%%%%%%%%%%%%%%%%%%%%%%%%%%%%%%%%%%%%%

Given a simple contractible graph $G=G_\tau$ and $p\in \NN$, we can look for 
the threshold $\lambda_p$ such that the random graph $X$ contains a path of length $p$,
whenever $\lambda>\lambda_p$.  We can proceed exactly
as in Theorem \ref{teoromitov}, with the simplification that the space 
$\omega_1^{\NN^{[k]}}$ is replaced by $p^{\NN^{[k]}}$, and obtain the following characterization of 
the threshold $\lambda_p$:

\begin{equation*}
\lambda_p:=\sup\left\{\inf_{v\in V_G=\NN^{[k]}}\mu(X_{v}):\ X \ \textrm{random graph without paths of length $p$}\right\}.
\end{equation*}

\bprop\label{prolp}
Let $p,k\in\NN$ and let $\tau\in\E_k$. 
We have 
%$\inf_{v\in \NN^{[k]}}m(A_v)\le \lambda_p$, where 
\begin{eqnarray*}
\lambda_p &=& \sup_{m\in \mathcal M_1(p^{\NN^{[k]}})}\inf_{v\in \NN^{[k]}}m(A_v)
\\
&=& \sup_{g:[0,1]^{\Sigma_\tau^{[\le k]}}\to p}
\left\vert\left\{ y\in [0,1]^{S^{[\le k]}}:
g(y|_{\Sigma_\tau^{[\le k]}})>g((\tau|_{\Sigma_\tau^{[\le k]}})^*(y))\right\}\right\vert .
\end{eqnarray*}
In particular, the random graph $X$ has a path of length $p$ if
\begin{equation}\label{hhhfff}
\inf_{v\in V_G}\mu(X_{v})> \lambda_p\,.
\end{equation}
On the other hand, if $\lambda<\lambda_p$ we can find $X$ and $\omega\in \Om$ such that $X(\omega)$ has no paths of length $p$,
and $\mu(X_{v})\ge \lambda$ for all $v\in V_G$.
\eprop

Notice that from Proposition \ref{prolp} it follows that 
$\lambda_p$ is an increasing function of $p$ and
$\lim_{p\to\infty} \lambda_p = \lambda_G$.

When $G=G_k$ we are able to explicitly compute the value of the threshold $\lambda_p$, which we denote by $\lambda_{p,k}$.

\bprop\label{profin}
Let $p,k\in\NN$. Then
\begin{equation*}%\label{eqlp}
\left( 1 - \frac{1}{k-1}\left\lceil\frac{k-1}{p}\right\rceil\right)\left( 1-\frac{1}{k}\right) \le \lambda_{p,k} \le
\left( 1 - \frac{1}{k-1}\left\lfloor\frac{k-1}{p}\right\rfloor\right)\left( 1-\frac{1}{k}\right).
\end{equation*}
In particular, if $(k-1)$ is a multiple of $p$ we get
\begin{equation}\label{eqlpp}
\lambda_{p,k} = \left( 1 - \frac 1 p\right)\left( 1-\frac{1}{k}\right).
\end{equation}
\eprop

\bp
As in Example \ref{exshift} we have $S=\{ 1,\ldots, k\}$,  
$\tau=s|^{}_S$ where $s$ is the shift map, and $\Sigma_\tau=\{ 1,\ldots, k-1\}$ 
(we set $\Sigma=\emptyset$ if $k=1$). Let us consider the function $h:[0,1]^{k-1}\to p$ defined as
\[
h(y_1,\ldots,y_{k-1})=\bar j-1 \quad ({\rm mod}\,p), 
\]
where the index $\bar j$ is such that $x_{\bar j}=\max_{1\le i\le k-1}x_i$, as in the proof of Lemma \ref{lemvar}.
We then have 
\begin{eqnarray*}
\lambda_{p,k}&=& \sup_{g:[0,1]^{\Sigma_\tau^{[\le k-1]}}\to p}\, 
\left\vert\left( \left\{ y\in [0,1]^{S^{[\le k-1]}}: g(y|_{\Sigma_\tau^{[\le k-1]}})
>g((\tau|_{\Sigma_\tau})^*(y))\right\}\right)\right\vert 
\\
&=&\sup_{g:[0,1]^{k-1}\to p}
\left\vert\left\{ y\in [0,1]^{k}: g(y_1,\ldots,y_{k-1})>g(y_2,\ldots,y_{k})\right\}\right\vert
\\
&\ge& 
\left\vert\left\{ y\in [0,1]^{k}: h(y_1,\ldots,y_{k-1})>h(y_2,\ldots,y_{k})\right\}\right\vert
\\
&\ge&
1 - \frac{1}{k}\left( 1+\left\lceil\frac{k-1}{p}\right\rceil\right)= 
\left( 1 - \frac{1}{k-1}\left\lceil\frac{k-1}{p}\right\rceil\right)\left( 1-\frac{1}{k}\right).
\end{eqnarray*}
% 
%We thus proved the first inequality in \eqref{eqlp}. \\
In order to prove the opposite inequality, for all functions $g:[0,1]^{\Sigma_\tau^{[\le k-1]}}\to p$ we let 
\[
A_g := \left\{ y\in [0,1]^{S^{[\le k-1]}}:\ g(y|_{\Sigma_\tau^{[\le k-1]}})>g((\tau|_{\Sigma_\tau})^*(y))\right\}.
\]
Let $\tilde\tau:\,S=\{ 1,\ldots,k\}\to S$ be the $k$-periodic function defined as 
$\tilde\tau(i)=i+1$ if $i< k$, and $\tilde\tau(k)=1$. We have
\[
k\left( 1-\vert A_g\vert\right) = 
k\vert A_g^c\vert \ge 
\int_{S^{[\le k-1]}} \sum_{j=0}^{k-1} \chi^{}_{\left(({\tilde{\tau}}^*)^j A_g\right)^c}(x)\,dx
\ge \left\lceil\frac{k}{p}\right\rceil = 1 + \left\lfloor \frac{k-1}{p}\right\rfloor
%\ge 1 + \left\lfloor \frac{k-1}{p}\right\rfloor\,,
\]
where $\chi_A$ denotes the characteristic function of the set $A\subset S^{[\le k-1]}$.
It then follows
\[
\vert A_g\vert \le \left( 1 - \frac{1}{k-1}\left\lfloor\frac{k-1}{p}\right\rfloor\right)\left( 1-\frac{1}{k}\right).
\]
\ep
{}From Proposition \ref{profin}, for all $k$ odd we get
\[
\lambda_{2,k} = \frac{k-1}{2k}\,.
\]
In particular $\lambda_{2,5}=2/5$, thus confirming a conjecture made in \cite{TW:98}.
We point out that in \cite{TW:98} it is also shown that $\lambda_{2,6}>5/12$,
so that \eqref{eqlpp} cannot hold for all couples $(p,k)$.

%%%%%%%%%%%%%%%%%%%%%%%%%%%%%%%%%%%%%%%%%%%%%%%%%%%%%%%%%%%%%%%%%%%
\subsection{Graphs with random edges}\label{randomed}
%%%%%%%%%%%%%%%%%%%%%%%%%%%%%%%%%%%%%%%%%%%%%%%%%%%%%%%%%%%%%%%%%%%

In this final section we discuss a variant of Problem \ref{problemv} for graphs with random edges.
More precisely, given a graph $G$ with vertex set $\NN^{[k]}$, 
we associate to each edge $e\in E_G$ a measurable set $X_e\subseteq \Om$,
where $(\Om, \mu)$ is a given probability space. As before, the random graph
can be equivalently defined by means of a $\mu$-measurable function 
$X: \Om \to 2^{E_G}$, such that $X_e := 
\{\omega\in \Om:\, e \in X(\omega)\}$ for all $e\in E_G$. 
In this setting, Problem \ref{problemv} becomes:
\bproblem \label{problemvbis} 
For all $e\in E_G$ let $X_e$ be a measurable subset 
of $(\Om, \mu)$, with $\mu(X_{e})\ge \lambda\in [0,1]$.
We ask for which values of $\lambda$
there exists an infinite sequence of vertices $v_i$ of $G$ 
such that $(v_i,v_{i+1})\in E_G$ for all $i\in\NN$ 
and $\bigcap_{i\in \NN} X_{v_i,v_{i+1}}$ is non-empty.
\eproblem
As above, answering to this question amounts to computing the 
threshold
\[
\widetilde\lambda_G:=\sup\left\{\inf_{e\in E_G}\mu(X_{e}):\ X \ \textrm{random graph without infinite paths}\right\}.
\]

Reasoning as in Section \ref{sectrash}, we let 
$X_v:=\cup_{v': (v,v')\in E_G}X_{v,v'}$, $Y_v\subset X_v$
be the subset of all $\omega$ such that 
$X(\omega)$ contains an infinite path starting from $v$, and 
the map $\phi:\Om \to (\omega_1+1)^{\NN^{[k]}}$ be defined as 
\[
\phi(\omega)_v = \left\{ \begin{array}{ll}
\displaystyle\sup_{\{v'\in \NN^{[k]}:\, (v,v')\in E_G,\,\omega\in X_{v,v'}\}} \phi(\omega)_{v'}+1
& {\rm if\ } \omega\in X_v\setminus Y_v,
\\
0 & {\rm if\ } \omega\not\in X_v
\\
\omega_1
& {\rm if\ } \omega\in Y_v.
\end{array}\right. 
\]
As before, if $X$ does not contain infinite paths then $\phi: \Om \to \omega_1^{\NN^{[k]}}$ is essentially bounded 
and 
$$
\phi(X_{v,v'})\subset A_{v,v'} := 
\left\{ x\in \omega_1^{\NN^{[k]}}:\, x_v>x_{v'}\right\} \qquad {\rm for\ all}\ (v,v')\in E_G.
$$
We now state the analog of Theorem \ref{teoromitov} and Proposition \ref{teoromitogen} in this setting. 

\bt \label{teoromitoedge}
Letting $G=G_\C$ be a contractible graph with $V_G=\NN^{[k]}$ and $\C\subset\E_k$, 
we have 
\begin{eqnarray*}%\label{mmvgen}
\widetilde \lambda_G &=& 
\sup_{m\in \mathcal M_{\rm c}(\omega_1^{\NN^{[k]}})}\inf_{(v,v')\in E_G}m(A_{v,v'})
\\
&=& \sup_{f:[0,1]^{S^{[\le k]}}\to [0,1]}
\left\vert \left\{ x\in [0,1]^{\NN^{[\le k]}}:\, 
f(x|_{S^{[\le k]}})>f(\tau^*(x))
\ \forall \tau\in\C\right\}\right\vert.
\end{eqnarray*}
When $G=G_\tau$ is a simple graph, from Lemma \ref{lemvar} it also follows
\[
\widetilde\lambda_G = 1 -\frac{1}{w(\tau)}\,.
\]
In particular, 
the random graph $X$ has an infinite path if
\begin{equation*}
\lambda:=
\inf_{v\in V_G}\mu(X_{v})> \widetilde\lambda_G\,.
\end{equation*}
On the other hand, if $\lambda<\widetilde\lambda_G$ we can find $X$ and $\omega\in\Om$ such that $X(\omega)$ has no infinite paths
and $\mu(X_{e})\ge \lambda$ for all $e\in E_G$.
\et

We point out that the proof of Theorem \ref{teoromitoedge} is analogous to the proof of 
Theorem \ref{teoromitov}, and in fact it is slightly shorter, 
due to the fact that the sets $A_{v,v'}$ have a simpler definition than the corrsponding sets $A_v$.

%%%%%%%%%%%%%%%%%%%%%%%%%%%%%%%%%%%%%%%%%%%%%%%%%%%%%%%%%%%%%%%%%%%%%%%%%%%%%%%%%%%%%%%%%%%

\end{document}